\documentclass{amsart}

\usepackage[english]{babel}
\usepackage[utf8]{inputenc}
\usepackage{amsmath,amssymb,amsthm}
\usepackage{tikz}
\usetikzlibrary{patterns,shapes,decorations.pathmorphing}
\usepackage{todonotes}
\usepackage{hyperref,graphicx}
\usepackage[backend=biber]{biblatex}
\newlength{\unit}
\newcommand{\DyckGrid}[1]{%
  \draw[step=\unit, thin] (0,0) grid (#1,#1);
  \draw[dotted] (0,0) -- (#1,#1);%
  \foreach \x in {1,...,#1}%
  {
    \node at (\x-0.5, -0.5) {\x};
    \node at (-0.5, \x-0.5) {\x};
  }
}

\tikzset{%
  MotzkinNode/.style={circle,minimum size=3pt,inner sep=0pt,fill=black,draw},
  MotzkinRed/.style={red, line width=1pt, decorate, decoration={snake, segment length=5pt, amplitude=1pt}},
  MotzkinBlue/.style={blue, line width=1pt, dashed},
  MotzkinBlack/.style={black, line width=1pt},
  DyckWidth/.style={rounded corners=1, line width=1.5pt}
}
\newcommand\open{\begin{tikzpicture}[x=\unit, y=\unit]
 \draw[dotted] (0,0)--(1,1);
 \draw[MotzkinBlack] (.5,1)--(.5,.5)--(1,.5);
\end{tikzpicture}}
\newcommand\close{\begin{tikzpicture}[x=\unit, y=\unit]
 \draw[dotted] (0,0)--(1,1);
 \draw[MotzkinBlack] (.5,0)--(.5,.5)--(0,.5);
\end{tikzpicture}}
\newcommand\ubounce{\begin{tikzpicture}[x=\unit, y=\unit]
 \draw[dotted] (0,0)--(1,1);
 \draw[MotzkinBlue] (0,.5)--(.5,.5)--(.5,1);
\end{tikzpicture}}
\newcommand\lbounce{\begin{tikzpicture}[x=\unit, y=\unit]
 \draw[dotted] (0,0)--(1,1);
 \draw[MotzkinRed] (.5,0)--(.5,.5);
 \draw[MotzkinRed] (.5,.5)--(1,.5);
\end{tikzpicture}}
\newcommand\fixed{\begin{tikzpicture}[x=\unit, y=\unit]
 \draw[dotted] (0,0)--(1,1);
 \node at (.5,.5) {$\times$};
\end{tikzpicture}}
\newcommand\up{\begin{tikzpicture}[x=\unit, y=\unit]
    \draw[transparent] (0,0)--(1,1);
    \node[MotzkinNode] (a) at (0,0) {};
    \node[MotzkinNode] (b) at (1,1) {};
    \draw[MotzkinBlack] (a) -- (b);
\end{tikzpicture}}
\newcommand\down{\begin{tikzpicture}[x=\unit, y=\unit]
    \draw[transparent] (0,0)--(1,1);
    \node[MotzkinNode] (a) at (0,1) {};
    \node[MotzkinNode] (b) at (1,0) {};
    \draw[MotzkinBlack] (a) -- (b);
\end{tikzpicture}}
\newcommand\levelzero{\begin{tikzpicture}[x=\unit, y=\unit]
    \draw[transparent] (0,0)--(1,1);
    \node[MotzkinNode] (a) at (0,.5) {};
    \node[MotzkinNode] (b) at (1,.5) {};
    \draw[MotzkinBlack] (a) -- (b);
\end{tikzpicture}}
\newcommand\levelred{\begin{tikzpicture}[x=\unit, y=\unit]
    \draw[transparent] (0,0)--(1,1);
    \node[MotzkinNode] (a) at (0,.5) {};
    \node[MotzkinNode] (b) at (1,.5) {};
    \draw[MotzkinRed] (a) -- (b);
\end{tikzpicture}}
\newcommand\levelblue{\begin{tikzpicture}[x=\unit, y=\unit]
    \draw[transparent] (0,0)--(1,1);
    \node[MotzkinNode] (a) at (0,.5) {};
    \node[MotzkinNode] (b) at (1,.5) {};
    \draw[MotzkinBlue] (a) -- (b);
\end{tikzpicture}}
\def\MotzkinU{-- ++(1,1) node[MotzkinNode]{}}
\def\MotzkinD{-- ++(1,-1) node[MotzkinNode]{}}
\def\MotzkinL{-- ++(1,0) node[MotzkinNode]{}}

\definecolor{darkblue}{rgb}{0,0,0.7} 
\newcommand{\Dfn}[1]{\emph{\color{darkblue} #1}} 

\newtheorem{thm}{Theorem}
\newtheorem{cnj}[thm]{Conjecture}
\newtheorem*{cnj*}{Conjecture}
\newtheorem{lem}[thm]{Lemma}

\theoremstyle{definition}

\title[Double deficiencies of Dyck paths]{Double deficiencies of Dyck paths via\\ the Billey-Jockusch-Stanley bijection}

\date{\today}

\author[M.~Rubey]{Martin Rubey}
\author[C.~Stump]{Christian Stump}

\address[M.~Rubey]{Fakult\"at f\"ur Mathematik und Geoinformation, TU Wien, Austria}
\email{Martin.Rubey@tuwien.ac.at}

\address[C.~Stump]{Fakult\"at f\"ur Mathematik, Otto-von-Guericke Universit\"at Magdeburg, Germany}
\email{Christian.Stump@ovgu.de}

\begin{document}

\begin{abstract}
  We prove a recent conjecture by Ren\'e Marczinzik involving certain
  statistics on Dyck paths that originate in the representation
  theory of Nakayama algebras of a linearly oriented quiver.  We do
  so by analysing the effect of the Billey-Jockusch-Stanley bijection
  between Dyck paths and $321$-avoiding permutations on these
  statistics, which was suggested by the result of a query issued to
  the online database FindStat.
\end{abstract}

\maketitle

\section{Introduction}

This paper serves two purposes.  The first is to demonstrate the power
of the online database FindStat~\cite{FindStat} to help with
explaining and recognising combinatorial parameters (also known as
\lq combinatorial statistics\rq) which occur in perhaps surprising
locations.

The second purpose, achieving the first, is to prove a refinement of
the combinatorial part of the following conjecture of Ren\'e
Marczinzik~\cite{Mar2017}:
\begin{cnj*}
  The number of $2$-Gorenstein algebras which are
  Nakayama algebras with $n$~simple modules and have an oriented line
  as associated quiver equals the number of Motzkin
  paths\footnote{\url{www.oeis.org/A001006}} of length~$n$.

  Moreover, the number of such algebras having the double centraliser
  property with respect to a minimal faithful projective-injective
  module equals the number of Riordan
  paths\footnote{\url{www.oeis.org/A005043}}, that is, Motzkin paths
  without level-steps at height zero, of length $n$.
\end{cnj*}

Let us stress that we do not attempt to explain the algebraic
significance of this conjecture.  Indeed, the part of reducing
it to an enumerative statement in combinatorics,
reproduced as
Conjecture~\ref{conj:1} below, is Marczinzik's achievement.

\section{Combinatorial background}

Consider a square array with columns labelled~$1$ through~$n$ from left to right and rows labelled~$1$ through~$n$ from bottom to top.
A \Dfn{Dyck path of semilength~$n$} is a lattice path with north and east steps running along the edges of the array, starting at the lower left corner,  ending at the upper right corner, and never going below the diagonal $y=x$.
We refer to Figure~\ref{fig:Dyck} for an illustration of our conventions.  Since all of the notions defined below depend on a Dyck path, we do not indicate the Dyck path in the notation to avoid clutter.

In the following we use two variants of the area sequence associated with a Dyck paths: the \Dfn{row-area sequence} $(r_0,r_1,\dots,r_n)$ is obtained by setting $r_0 = -1$ and $r_k$, for $1 \leq k \leq n$, to the number of full squares in the row of the~$k$-th north step between the path and the main diagonal.
For example, the row-area sequence of the Dyck path in Figure~\ref{fig:Dyck} is
\[
(-1,0,1,1,2,3,2,3,4,5,6,7,8,4,4,3,2,2).
\]
Similarly, the \Dfn{column-area sequence} $(c_1,\ldots,c_{n+1})$ is obtained by setting $c_k$, for $1 \leq k \leq n$, to the number of full squares in the column of the~$k$-th east step between the path and the main diagonal.  Additionally, we set $c_{n+1} = -1$.
In the example in Figure~\ref{fig:Dyck}, the column-area sequence is
\[
(1,3,2,8,7,6,5,4,4,4,3,3,2,2,2,1,0,-1).
\]
The additional $-1$ at the beginning of the row- and at the end of the column-area sequence can be interpreted as prepending a north step and appending an east step to the Dyck path, without shifting the main diagonal.
It turns out that several properties below
are easier to describe with this convention.

Finally, a \Dfn{valley} of a Dyck path is an east step directly followed by a north step.  In terms of the array, it is the cell enclosed by the two steps.  Explicitly, if the east step of the valley is the $k$-th east step of the path, and the north step is the $\ell$-th north step of the path, the position of the valley is $(k,\ell)$.

\section{Marczinzik's conjecture and its refinement}

The following notions on Dyck paths are due to Marczinzik and originate in the representation theory of Nakayama algebras of a linearly oriented quiver on~$n$ vertices.
We refer to the MathOverflow discussion~\cite{Mar2017} for further background.
For a Dyck path~$D$ with row-area sequence $(r_0,r_1,\dots,r_n)$ and column-area sequence $(c_1,\ldots,c_{n+1})$,
\begin{itemize}
  \item \Dfn{$\mathcal D$} is the set of indices $k$ with $c_{k+1} = c_k - 1$,
  \item \Dfn{$\mathcal F$}, is the set of indices k with
    $r_{k+1+c_{k+1}} = r_{k-1}+c_{k+1}+2$, and
  \item \Dfn{$\mathcal N$} be the set of rows which do not contain a valley.
\end{itemize}

We can now state Marczinzik's conjecture.
Recall that a \Dfn{Motzkin path of length~$n$} is a lattice path from $(0,0)$ to $(n,0)$ consisting of up-steps $(1,1)$, down-steps $(1,-1)$ and level-steps $(1,0)$ that never goes below the $x$-axis.
\begin{cnj}[Marczinzik~\cite{Mar2017}]
  \label{conj:1}
  The number of Dyck paths such that
  $\mathcal N\cap \mathcal D$ is contained in $\mathcal F$ equals the number of
  Motzkin paths of length~$n$.

  Moreover, the number of Dyck paths such that
  $\mathcal N\cap\mathcal D$ is empty equals the number of Riordan
  paths of length~$n$, that is, Motzkin paths without level-steps at
  height zero.
\end{cnj}
The case $n=3$ is illustrated in Figure~\ref{fig:minimal}.

\medskip

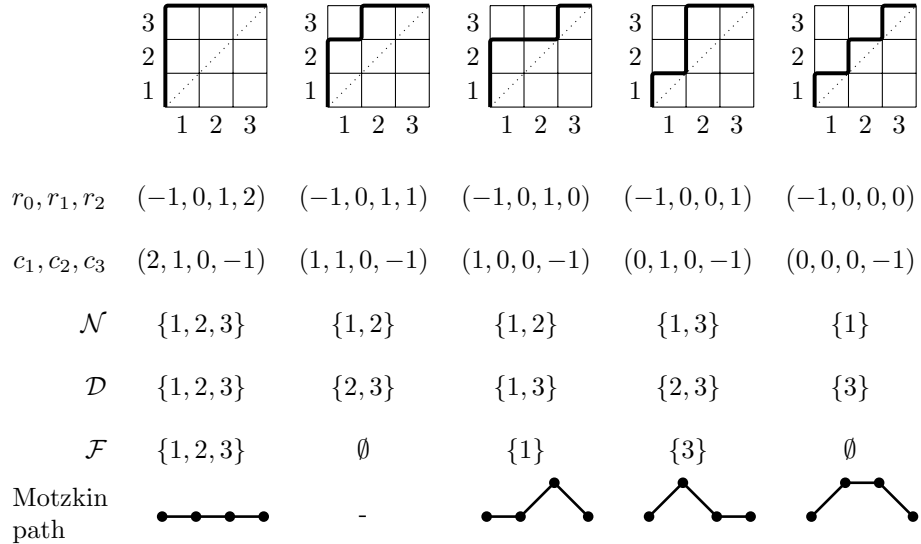
\begin{figure}
  \centering%
  \setlength{\unit}{0.45cm} 
  \renewcommand{\arraystretch}{2}
  \begin{tabular}{rccccc}
    &
      \begin{tikzpicture}[x=\unit, y=\unit]
        \DyckGrid{3}%
        \draw[DyckWidth]%
        (0,0)--(0,3)--(3,3);%
      \end{tikzpicture}
    &
      \begin{tikzpicture}[x=\unit, y=\unit]
        \DyckGrid{3}%
        \draw[DyckWidth]%
        (0,0)--(0,2)--(1,2)--(1,3)--(3,3);%
      \end{tikzpicture}
    &
      \begin{tikzpicture}[x=\unit, y=\unit]
        \DyckGrid{3}%
        \draw[DyckWidth]%
        (0,0)--(0,2)--(2,2)--(2,3)--(3,3);%
      \end{tikzpicture}
    &
      \begin{tikzpicture}[x=\unit, y=\unit]
        \DyckGrid{3}%
        \draw[DyckWidth]%
        (0,0)--(0,1)--(1,1)--(1,3)--(3,3);%
      \end{tikzpicture}
    &
      \begin{tikzpicture}[x=\unit, y=\unit]
        \DyckGrid{3}%
        \draw[DyckWidth]%
        (0,0)--(0,1)--(1,1)--(1,2)--(2,2)--(2,3)--(3,3);%
      \end{tikzpicture}
    \\
    $r_0,r_1,r_2$ & $(-1,0,1,2)$ & $(-1,0,1,1)$ & $(-1,0,1,0)$ & $(-1,0,0,1)$ & $(-1,0,0,0)$\\
    $c_1,c_2,c_3$ & $(2,1,0,-1)$ & $(1,1,0,-1)$ & $(1,0,0,-1)$ & $(0,1,0,-1)$ & $(0,0,0,-1)$\\
    $\mathcal N$  & $\{1,2,3\}$  & $\{1,2\}$    & $\{1,2\}$    & $\{1,3\}$    & $\{1\}$\\
    $\mathcal D$  & $\{1,2,3\}$  & $\{2,3\}$    & $\{1,3\}$    & $\{2,3\}$    & $\{3\}$\\
    $\mathcal F$  & $\{1,2,3\}$  & $\emptyset$  & $\{1\}$      & $\{3\}$      & $\emptyset$\\
    \parbox{1.25cm}{Motzkin\\path}
    &
      \hfill
      \begin{tikzpicture}[x=\unit, y=\unit]
        \draw[MotzkinBlack] (0,0) node[MotzkinNode]{} \MotzkinL\MotzkinL\MotzkinL;
      \end{tikzpicture}
    &
      -
    &
      \hfill
      \begin{tikzpicture}[x=\unit, y=\unit]
        \draw[MotzkinBlack] (0,0) node[MotzkinNode]{} \MotzkinL\MotzkinU\MotzkinD;
      \end{tikzpicture}
    &
      \hfill
      \begin{tikzpicture}[x=\unit, y=\unit]
        \draw[MotzkinBlack] (0,0) node[MotzkinNode]{} \MotzkinU\MotzkinD\MotzkinL;
      \end{tikzpicture}
    &
      \hfill
      \begin{tikzpicture}[x=\unit, y=\unit]
        \draw[MotzkinBlack] (0,0) node[MotzkinNode]{} \MotzkinU\MotzkinL\MotzkinD;
      \end{tikzpicture}
  \end{tabular}
  \caption{The case $n=3$.}
  \label{fig:minimal}
\end{figure}

Upon seeing this conjecture, our immediate reaction was to transform the condition into a \lq combinatorial statistic\rq\ on Dyck paths, and query the online database FindStat.  Indeed, this approach was doubly successful:  first, the result of the database query, reproduced in the conjecture below, is a substantial common refinement of both parts of Conjecture~\ref{conj:1}, and makes a proof strategy suggest itself.  Second, the references linked in the result provide already all the tools we need.

\begin{cnj}
  \label{conj:2}
  For a Dyck path~$D$, an index $k$ is in $\mathcal F$ but not in
  $\mathcal N\cap\mathcal D$ if and only if it is a double
  deficiency\footnote{\url{www.findstat.org/St000732}} in the
  $321$-avoiding permutation associated with~$D$ using the
  Billey-Jockusch-Stanley
  bijection\footnote{\url{www.findstat.org/Mp00129}}.
\end{cnj}

This conjecture is verified in the following section.  In the final section, we use Sergi Elizalde's description~\cite{Elizalde2017} of the Foata-Zeilberger bijection between $321$-avoiding permutations and bicoloured Motzkin paths to deduce Conjecture~\ref{conj:1}.

\section{Dyck paths and 321-avoiding permutations}
\label{sec:dyck-paths-321}

\begin{figure}
  \centering%
  \setlength{\unit}{0.7cm} 
  \begin{tikzpicture}[x=\unit, y=\unit]
    \draw[step=\unit, thin] (0,0) grid (17,17);
    \draw[dotted] (0,0) -- (17,17);%
    \foreach \x in {1,...,17}%
    {
      \node at (\x-0.5, -0.5) {\x};
      \node at (-0.5, \x-0.5) {\x};
    }
    \foreach \x/\y in {1/3,2/1,3/6,4/2,5/4,6/5,7/7,8/13,9/14,10/8,11/15,12/9,13/16,14/17,15/10,16/11,17/12}%
    {
      \node at (\x-0.5,\y-0.5) {\huge$\times$};%
    }
    \draw[DyckWidth]%
    (0,0)--(0,2)--%
    (1,2)--(1,5)--%
    (3,5)--(3,12)--%
    (8,12)--(8,13)--%
    (9,13)--(9,14)--%
    (11,14)--(11,15)--%
    (13,15)--(13,16)--%
    (14,16)--(14,17)--(17,17);%
    \draw[<->, thick] (3+.1,5.5)--(6-.1,5.5) node[midway,below,rounded corners,fill=gray!20,inner sep=1] {\tiny $r_{k-1}+1$};
    \draw[<->, thick] (3+.1,11.5)--(11-.1,11.5) node[near start,below,rounded corners,fill=gray!20,inner sep=1] {\tiny $r_{k+1+c_{k+1}}$};;
    \draw[<->, thick] (7.5,7+.1)--(7.5,12-.1) node[midway,right,rounded corners,fill=gray!20,inner sep=1] {\tiny $c_{k+1}+1$};
    \draw[<->, thick, gray] (9+.1,13.5)--(14-.1,13.5); 
    \draw[<->, thick, gray] (14+.1,16.5)--(16-.1,16.5); 
    \draw[<->, thick, gray] (15.5,15+.1)--(15.5,17-.1); 
  \end{tikzpicture}
  \caption{A detailed example.}
  \label{fig:Dyck}
\end{figure}
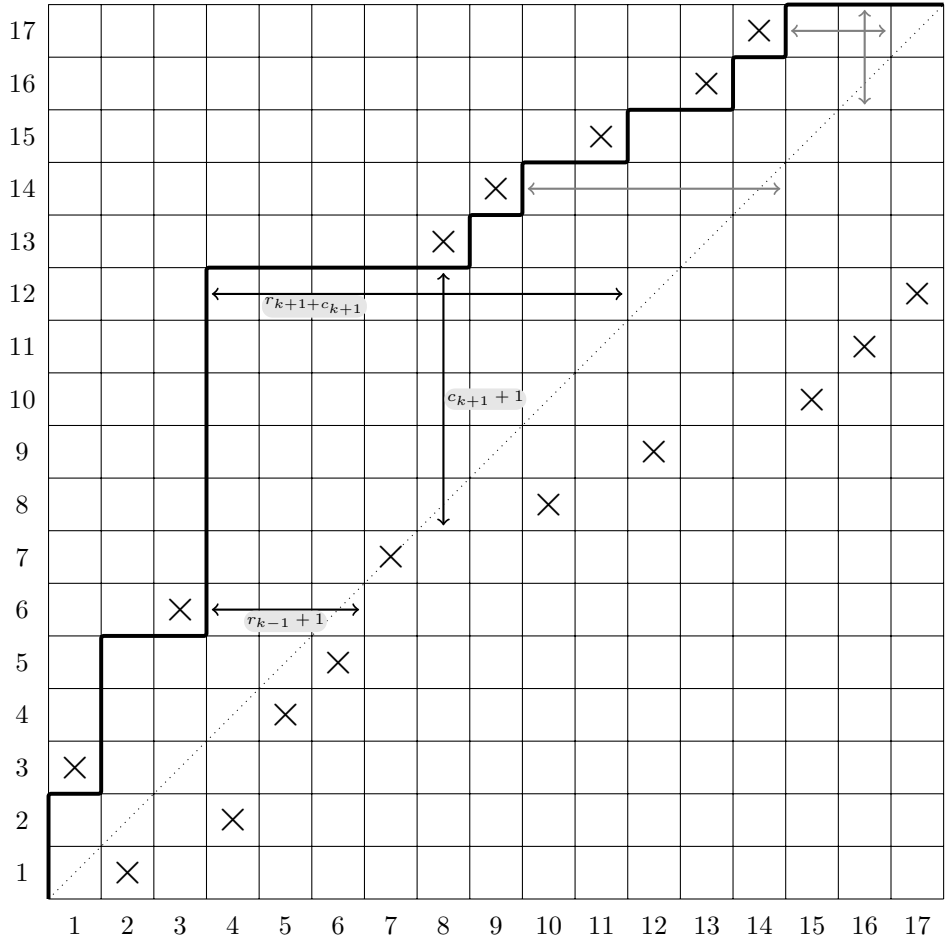

The Billey-Jockusch-Stanley bijection~\cite{BJS1993},
sending a Dyck path~$D$ of semi\-length $n$ to
a $321$-avoiding permutation~$\pi$ of the numbers $\{1,\dots,n\}$, goes as
follows: first, put crosses into the cells corresponding to the
valleys of~$D$.  Then, working from the left to the right, for each
column not yet containing a cross we put a cross into the lowest cell
whose row does not yet contain a cross.
This yields the permutation matrix of the permutation~$\pi$.

An example is given in Figure~\ref{fig:Dyck}, where the displayed
Dyck path of semilength~$17$ is sent to the permutation
\[
  [\ 3,1,6,2,4,5,7,13,14,8,15,9,16,17,10,11,12\ ]
\]
on $\{1,\ldots,17\}$, in one-line notation.

\medskip

For a permutation~$\pi$ of $\{1,\ldots,n\}$ and an index~$1 \leq k \leq n$, we say that~$k$ is
\begin{itemize}
  \item an \Dfn{excedance} if $\pi(k) > k$,
  \item a \Dfn{fixpoint} if $\pi(k) = k$,
  \item a \Dfn{deficiency} if $\pi(k) < k$, and
  \item a \Dfn{double deficiency} if $\pi(k) < k < \pi^{-1}(k)$.
\end{itemize}

Let us first record two general properties of the bijection.
\begin{lem}
\label{lem:BJS-exc-def}
  The crosses in the valleys of~$D$ are excedances of~$\pi$,
  whereas all others are fixpoints or deficiencies of~$\pi$.
\end{lem}
\begin{proof}
  The first statement is true because Dyck paths stay above the main
  diagonal.  To see the second statement, note that to the left of
  column~$k$ there are only $k-1$ crosses, so at least one of
  the bottom~$k$ rows cannot contain a cross to the left of column~$k$.
\end{proof}

\begin{lem}
\label{lem:BJS-fix}
  An index~$k$ is a fixpoint of~$\pi$ if and only if~$D$ does not have a valley in any position $(i,j)$ with $i \leq k$ and $j \geq k$.
\end{lem}
\begin{proof}
  Let~$k$ be a fixpoint of~$\pi$.
  The construction of~$\pi$ implies that column~$k$ does not contain a valley of~$D$.  Moreover, for every $\ell < k$ there is an $\ell' < k$ with $\pi(\ell') = \ell$.
  In other words $\pi$ restricts to a permutation of $\{1,\ldots,k-1\}$, implying that~$D$ does not have a valley in position $(i,j)$ for $i \leq k$ and $j \geq k$.
  As both implications in the argument are equivalences, the statement follows.
\end{proof}

Conjecture~\ref{conj:2} is now an immediate consequence of the following
three statements.
\begin{lem}
  \label{lem:dpoints}
  An index~$k$ is in $\mathcal D$ if and only if it is a fixpoint or a deficiency of~$\pi$.
\end{lem}
\begin{proof}
  By definition,~$k$ is in $\mathcal D$ if and only if there is no
  valley in column $k$.  Thus the claim is the statement of
  Lemma~\ref{lem:BJS-exc-def}.
\end{proof}

\begin{lem}
  \label{lem:fpoints}
  An index~$k$ is in $\mathcal F$ if and only if it is a fixpoint
  of~$\pi$.
\end{lem}
\begin{proof}
  Let us first remark that an index $k$ in $\mathcal F$ is an index
  for which~$D$ does not have a valley between the $(k-1)$-st north
  step and the $(k+1)$-st east step.

  This is best understood by looking at an example: in
  Figure~\ref{fig:Dyck} the index $k = 7$ is in $\mathcal F$ because
  \[
    r_6 + c_8 + 2 = 2 + 4 + 2 = 8 = r_{12} = r_{8 + c_8}.
  \]
  On the other hand, $k=15$ is not in $\mathcal F$, since
  \[
    r_{14} + c_{16} + 2 = 5 + 2 + 2 = 9 \neq 2 = r_{17} = r_{16 + c_{16}}.
  \]

  Suppose now that~$k$ is a fixpoint of~$\pi$.
  Then $k+1+c_{k+1}$ is the index of the row just below the $(k+1)$-st east step.
  The number of full squares in this row between the Dyck path and
  strictly to the left of column~$k$ is, by Lemma~\ref{lem:BJS-fix}, precisely $r_{k} = r_{k-1} + 1$.
  Moreover, the number of remaining full squares in this row, towards the main diagonal, is precisely $c_{k} = c_{k+1} + 1$.
  Observe that in this case, we are just rewriting the equality in
  the definition of $\mathcal F$ as $r_{k+c_k} = r_k + c_k$.

  On the other hand, if the cross in column~$k$ is not a fixpoint,
  there must be a valley to the left and above $(k,k)$, which entails
  that the number of full squares in row $k+1+c_{k+1}$ will be strictly
  smaller than $r_{k+1}+c_{k-1}+2$.
  Observe that in this case, it is not in general possible to rewrite the equality in the definition of $\mathcal F$ as above.  For example, this is the case with the index $k=6$ in Figure~\ref{fig:Dyck}.
\end{proof}

\begin{lem}
  \label{lem:doubledef}
  An index~$k$ is in~$\mathcal N$ if and only if $\pi^{-1}(k)$ is a fixpoint or a deficiency of~$\pi$.
\end{lem}
\begin{proof}
  By definition, $k \in \mathcal N$ if and only if row~$k$ does not contain a valley of~$D$.  This is the same as saying that the cross in row~$k$ is not to the left of the main diagonal, in symbols, $\pi^{-1}(k) \geq k$.
\end{proof}

\begin{proof}[Proof of Conjecture~\ref{conj:2}]
  This is a direct consequence of Lemmas~\ref{lem:dpoints},~\ref{lem:fpoints}, and~\ref{lem:doubledef}.
\end{proof}

\section{321-avoiding permutations and bicoloured Motzkin paths}
\label{sec:Motzkin}

\begin{figure}
  \centering%
  \setlength{\unit}{0.7cm} 
  \begin{tikzpicture}[x=\unit, y=\unit]
    \draw[step=\unit, thin] (0,0) grid (17,17);
    \draw[dotted] (0,0) -- (17,17);%
    \foreach \x in {1,...,17}%
    {
      \node at (\x-0.5, -0.5) {\x};
      \node at (-0.5, \x-0.5) {\x};
    }
    \foreach \x/\y in {1/3,2/1,3/6,4/2,5/4,6/5,7/7,8/13,9/14,10/8,11/15,12/9,13/16,14/17,15/10,16/11,17/12}%
    {
      \node at (\x-0.5,\y-0.5) {\huge$\times$};%
    }
    \draw[DyckWidth]%
    (0,0)--(0,2)--%
    (1,2)--(1,5)--%
    (3,5)--(3,12)--%
    (8,12)--(8,13)--%
    (9,13)--(9,14)--%
    (11,14)--(11,15)--%
    (13,15)--(13,16)--%
    (14,16)--(14,17)--(17,17);%
    \foreach \x/\y in {1/3,6/5,8/13,9/14,11/15,15/10,16/11,17/12}%
    {
      \draw[MotzkinBlack] (\x-0.5,\y-0.5) -- (\x-0.5,\x-0.5);
    }
    \foreach \x/\y in {2/1,3/6,10/8,11/15,12/9,13/16,14/17,16/11}%
    {
      \draw[MotzkinBlack] (\x-0.5,\y-0.5) -- (\y-0.5,\y-0.5);
    }
    \foreach \x/\y in {2/1,3/6,4/2,5/4,10/8,12/9,13/16,14/17}%
    {
      \ifthenelse {\x<\y}
      {
        \draw[MotzkinBlue] (\x-0.5,\y-0.5) -- (\x-0.5,\x-0.5);
      }
      {
        \draw[MotzkinRed] (\x-0.5,\y-0.5) -- (\x-0.5,\x-0.5);
      }
    }
    \foreach \x/\y in {1/3,4/2,5/4,6/5,8/13,9/14,15/10,17/12}%
    {
      \ifthenelse {\x<\y}
      {
        \draw[MotzkinBlue] (\x-0.5,\y-0.5) -- (\y-0.5,\y-0.5);
      }
      {
        \draw[MotzkinRed] (\x-0.5,\y-0.5) -- (\y-0.5,\y-0.5);
      }
    }
    %
    \foreach \x/\h in {1/0,8/0,9/1,11/2}%
    {
      \node[MotzkinNode] (a) at (\x-1.0,\h-4.5) {};
      \node[MotzkinNode] (b) at (\x,\h-3.5) {};
      \draw[MotzkinBlack] (a) -- (b);
    }
    \foreach \x/\h in {6/1,15/3,16/2,17/1}%
    {
      \node[MotzkinNode] (a) at (\x-1.0,\h-4.5) {};
      \node[MotzkinNode] (b) at (\x,\h-5.5) {};
      \draw[MotzkinBlack] (a) -- (b);
    }
    \foreach \x/\h in {2/1,4/1,5/1,10/2,12/3}%
    {
      \node[MotzkinNode] (a) at (\x-1.0,\h-4.5) {};
      \node[MotzkinNode] (b) at (\x,\h-4.5) {};
      \draw[MotzkinRed] (a) -- (b);
    }
    \foreach \x/\h in {3/1,13/3,14/3}%
    {
      \node[MotzkinNode] (a) at (\x-1.0,\h-4.5) {};
      \node[MotzkinNode] (b) at (\x,\h-4.5) {};
      \draw[MotzkinBlue] (a) -- (b);
    }
    \node[MotzkinNode] (a) at (7-1.0,-4.5) {};
    \node[MotzkinNode] (b) at (7,-4.5) {};
    \draw[MotzkinBlack] (a) -- (b);
  \end{tikzpicture}
  \caption{The associated bicoloured Motzkin path.}
  \label{fig:Motzkin}
\end{figure}
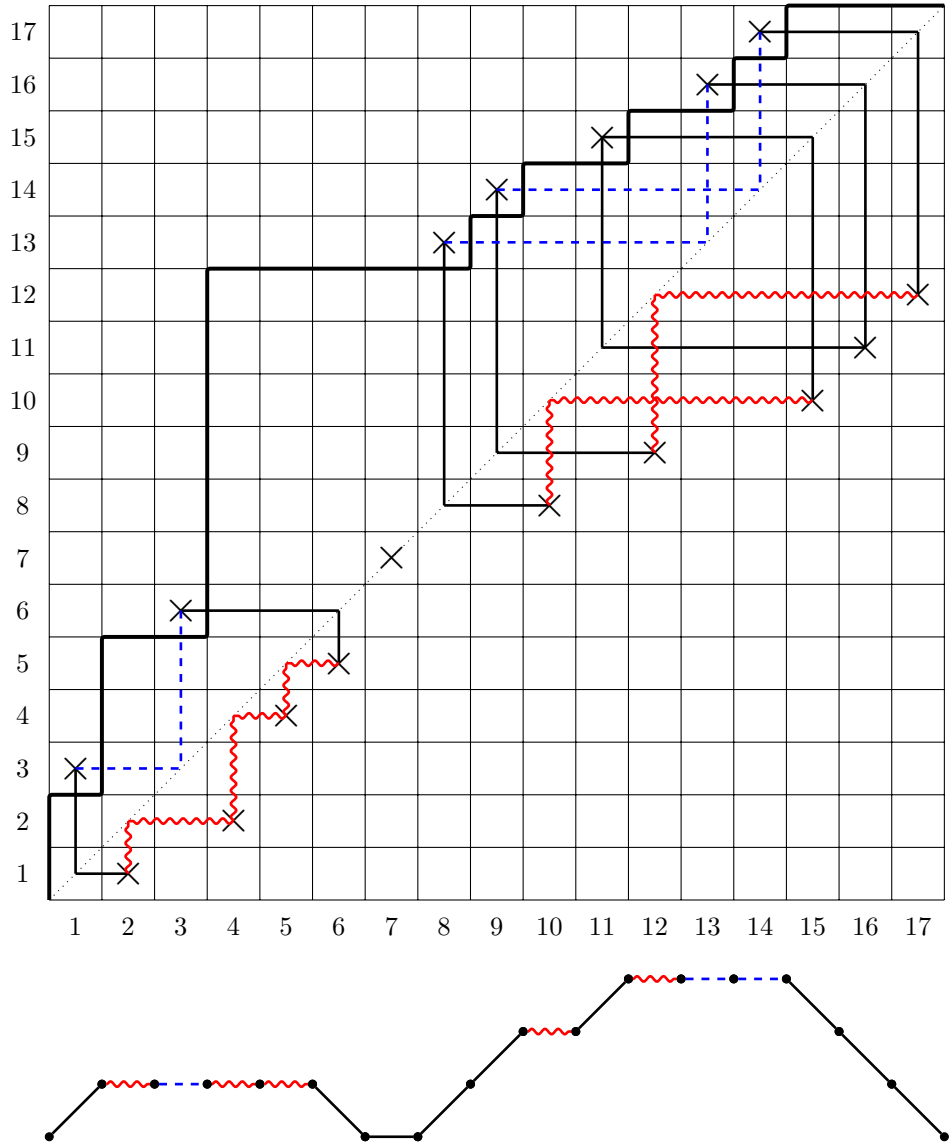

Recall that a \Dfn{bicolored Motzkin path} is a Motzkin path where level-steps not at height zero come in two colors, say blue and red.
To prove Conjecture~\ref{conj:1}, we follow Elizalde's description~\cite{Elizalde2017} of
the Foata-Zeilberger bijection restricted to $321$-avoiding permutations and bicolored
Motzkin paths: from each cross in the array draw a horizontal and a
vertical line to the diagonal.  Then, looking at these lines as
emanating from the diagonal, there are five possibilities, which are
translated into up, down, level-steps at height zero, blue level-steps and red level-steps of the Motzkin path as follows:
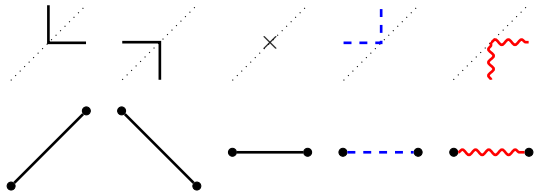
\[
  \setlength{\unit}{1cm} 
  \begin{array}{ccccc}
    \open & \close & \fixed & \ubounce & \lbounce\\[6pt]
    \up & \down & \levelzero & \levelblue & \levelred
  \end{array}
\]
We can now deduce Conjecture~\ref{conj:1} from Conjecture~\ref{conj:2}:

\begin{proof}[Proof of Conjecture~\ref{conj:1}]
  Let $D$ be a Dyck path, $\pi$ the corresponding permutation and~$M$ the corresponding bicoloured Motzkin path.

  The construction of the Motzkin path is such that double
  deficiencies of~$\pi$ correspond to red level-steps of~$M$.
  Bicoloured Motzkin paths without any red level-steps are simply
  Motzkin path.  Together with Conjecture~\ref{conj:2} this implies
  the first part of Conjecture~\ref{conj:1}.

  By Lemma~\ref{lem:fpoints}, $\mathcal F$ is precisely the set of fixpoints of $\pi$, by Lemma~\ref{lem:dpoints} $\mathcal F$ is contained in $\mathcal D$, and by Lemma~\ref{lem:doubledef} $\mathcal F$ is contained in $\mathcal N$.
  Therefore, the statement that $\mathcal N$ does not contain $\mathcal D$ is equivalent to the statement that~$\pi$ does neither have double deficiencies nor fixpoints.
  The second part of the conjecture thus follows from the first, together with the observation that level-steps at height zero correspond to fixpoints of~$\pi$.
\end{proof}

\printbibliography
\end{document}